\documentclass[11pt]{amsart}

\usepackage[T1]{fontenc}
\usepackage{lmodern}
\usepackage{microtype}
\usepackage[margin=1.6in]{geometry}

\usepackage[
    backend=biber,
    style=alphabetic,
    sorting=nyt,
    giveninits=true,
    doi=false,
    url=false,
    isbn=false,
    maxbibnames=99,
    maxalphanames=99
]{biblatex}

\AtEveryBibitem{%
  \clearfield{month}%
  \clearfield{day}%
}

\usepackage{amsmath,amssymb,amsfonts,mathtools}
\usepackage{amscd}
\usepackage{latexsym}
\usepackage{mathrsfs}
\usepackage{accents}
\usepackage{esint}
\usepackage{braket}

\usepackage{graphicx}
\usepackage{enumitem}

\usepackage{xcolor}
\usepackage{aliascnt}
\usepackage[
  colorlinks=true,
  linkcolor=blue,
  citecolor=blue,
  urlcolor=blue
]{hyperref}
\usepackage[nameinlink,noabbrev]{cleveref}

\hypersetup{
  pdftitle={Genus zero FBMS},
  pdfauthor={Otis Chodosh and Matilde Gianocca}
}

\setlist{itemsep=3pt}
\allowdisplaybreaks
\numberwithin{equation}{section}

\theoremstyle{plain}

\newtheorem{theorem}{Theorem}[section]

\newaliascnt{proposition}{theorem}
\newtheorem{proposition}[proposition]{Proposition}
\aliascntresetthe{proposition}

\newaliascnt{lemma}{theorem}
\newtheorem{lemma}[lemma]{Lemma}
\aliascntresetthe{lemma}

\newaliascnt{corollary}{theorem}
\newtheorem{corollary}[corollary]{Corollary}
\aliascntresetthe{corollary}

\theoremstyle{remark}

\newaliascnt{remark}{theorem}
\newtheorem{remark}[remark]{Remark}
\aliascntresetthe{remark}

\crefname{theorem}{theorem}{theorems}
\Crefname{theorem}{Theorem}{Theorems}

\crefname{proposition}{proposition}{propositions}
\Crefname{proposition}{Proposition}{Propositions}

\crefname{lemma}{lemma}{lemmas}
\Crefname{lemma}{Lemma}{Lemmas}

\crefname{corollary}{corollary}{corollaries}
\Crefname{corollary}{Corollary}{Corollaries}

\crefname{remark}{remark}{remarks}
\Crefname{remark}{Remark}{Remarks}

\theoremstyle{plain}

\theoremstyle{definition}

\theoremstyle{remark}

\crefalias{theo}{theorem}
\crefalias{prop}{proposition}
\crefalias{lemm}{lemma}
\crefalias{coro}{corollary}
\crefalias{rema}{remark}

\newcommand{\BB}{\mathbb{B}}

\newcommand{\RR}{\mathbb{R}}

\newcommand{\R}{\mathbb{R}}

\DeclareMathOperator{\tr}{tr}

\let\oldmarginpar\marginpar
\renewcommand{\marginpar}[1]{%
  \oldmarginpar[\raggedleft\footnotesize #1]{\raggedright\footnotesize #1}}

\title{Free boundary minimal surfaces of genus zero}

\author[O. Chodosh]{Otis Chodosh}
\address{Department of Mathematics, Stanford University, Stanford, CA 94305, USA}
\email{ochodosh@stanford.edu}

\author[M. Gianocca]{Matilde Gianocca}
\address{Department of Mathematics, University of Chicago, Chicago IL 60637, USA}
\email{mgianocca@uchicago.edu}

\date{}

\begin{document}

\begin{abstract}
We prove that any embedded genus zero free boundary minimal surface in the ball has first Steklov eigenvalue $1$. This implies the uniqueness of the critical catenoid.
\end{abstract}

\maketitle

\section{Introduction}

Let $\Sigma\subset\BB^3$ be a smooth, compact, connected, properly embedded free-boundary minimal surface. Its Steklov eigenvalues are the numbers $\sigma$ for which there is a nonzero function satisfying
$$
\Delta f=0\quad\text{in }\Sigma,\qquad \partial_\eta f=\sigma f\quad\text{on }\partial\Sigma,
$$
where $\eta$ is the outward unit conormal. We prove that the first nonzero Steklov eigenvalue is equal to $1$ when $\Sigma$ has genus zero.

\begin{theorem}\label{thm:main}
If $\Sigma\subset\BB^3$ is a smooth properly embedded free-boundary minimal surface of genus zero, then $\sigma_1(\Sigma)=1$.
\end{theorem}

This confirms the conjecture of Fraser--Li \cite{FraserLi2014} in the genus zero case. By work  of Fraser--Schoen \cite[Theorem 6.6]{FraserSchoen2016}, \Cref{thm:main} implies

\begin{corollary}
    Up to rotations, the critical catenoid is the unique embedded free-boundary minimal annulus.
\end{corollary}

This may be seen as the free boundary analogue of the Lawson conjecture (uniqueness of the Clifford torus), as resolved by Brendle \cite{Brendle:lawson} using a two-point maximum principle argument.

\subsection{Related literature}

Besides the equatorial disk and the critical catenoid, many embedded free-boundary minimal surfaces in $\BB^3$ have been constructed by variational methods \cite{Ketover2016,CarlottoFranzSchulz2022,Chu2023,FranzKetoverSchulz2024,FranzSchulz2026}, gluing methods \cite{FolhaPacardZolotareva2017,KapouleasLi2021,KapouleasMcGrath2023,KapouleasWiygul2023,CarlottoSchulzWiygul2025a,CarlottoSchulzWiygul2025b,KapouleasZou2026}, and spectral methods \cite{FraserSchoen2016,KarpukhinStern2024}. In fact, every compact orientable surface with nonempty boundary occurs as an embedded free-boundary minimal surface in $\BB^3$ \cite{KarpukhinKusnerMcGrathStern2024}.

The connection with the Steklov problem was developed by Fraser--Schoen \cite{FraserSchoen2011,FraserSchoen2016}. The coordinate functions of a free-boundary minimal immersion are Steklov eigenfunctions with eigenvalue $1$, while extremal metrics for the normalized first Steklov eigenvalue give free-boundary minimal immersions by first eigenfunctions. Fraser--Li conjectured that every embedded free-boundary minimal surface in $\BB^3$ has $\sigma_1=1$ \cite{FraserLi2014}. Earlier cases under symmetry assumptions were obtained by McGrath and Seo \cite{McGrath2018,Seo2024,Seo2026}. Using the two-piece property of Lima--Menezes \cite{LimaMenezes2021}, Kusner--McGrath obtained further symmetric cases and the radial-graph property used below \cite{KusnerMcGrath2024}; related radial-projection estimates were obtained by McGrath--Zou \cite{McGrathZou2024}.

Nonrotational immersed free-boundary minimal annuli in $\BB^3$ were constructed in \cite{FernandezHauswirthMira2023,KapouleasMcGrath2022}. These annuli necessarily satisfy $\sigma_1<1$, since Fraser--Schoen proved that the critical catenoid is the only immersed free-boundary minimal annulus whose coordinate functions are first Steklov eigenfunctions \cite{FraserSchoen2016}.

Independently and around the same time, Parise and Zhu \cite{PZ26} obtained Corollary 1.2 by a different argument.

\subsection{Description of the strategy}

The min--max characterization of the first nonzero Steklov eigenvalue gives
$$
\sigma_1(\Sigma)=\inf_{\substack{f\in H^1(\Sigma),\,f|_{\partial\Sigma}\not\equiv0\\ \int_{\partial\Sigma}f\,ds=0}}\frac{\int_\Sigma|\nabla f|^2\,dA}{\int_{\partial\Sigma}f^2\,ds}.
$$
Consider the quadratic form
$$
Q(f,\psi)=\int_\Sigma\langle\nabla f,\nabla\psi\rangle\,dA-\int_{\partial\Sigma}f\psi\,ds.
$$
\Cref{thm:main} is equivalent to the claim that $Q$ has at most one negative direction.

When $\Sigma$ is a radial graph (all non-flat embedded genus zero free boundary minimal surfaces are radial graphs) test functions on $\Sigma$ can be approximated by functions on $S^2$ composed with the radial projection. More precisely, writing $r=|X|$ and $n=X/|X|$, one has
$$
\overline{\{r(\phi\circ n):\phi\in C^\infty(S^2)\}}^{\,H^1(\Sigma)}=H^1(\Sigma).
$$
This allows us to study $Q$ through a quadratic form on the sphere. Indeed, for an explicit tensor $T$ on $S^2$,
$$
Q\bigl(r(\phi\circ n),r(\psi\circ n)\bigr)=q_T(\phi,\psi),
$$
where
$$
q_T(\phi,\psi)=\int_{S^2}\left(\langle T\nabla\phi,\nabla\psi\rangle-(\tr T)\phi\psi\right)d\omega.
$$
The tensor $T$ has the properties
$$
T=T^*,\qquad T\geq0,\qquad T\not\equiv0,\qquad \operatorname{div}T=0.
$$
Pretend that $T$ is smooth in the sequel (in reality we use a smoothing argument). 

Results in convex geometry (Minkowski's quadratic inequality) imply that every quadratic form $q_T$ associated with such a tensor has index one, see \Cref{thm:spherical}. We sketch (one) proof of this result here, closely following the work of Shenfeld--van Handel \cite{ShenfeldVanHandel2019}. The condition $\operatorname{div}T=0$ allows one to construct a positive function $u$ and a self-adjoint operator $L$ such that
$$
q_T(\phi,\psi)=-\langle\phi,L\psi\rangle,\qquad Lu=u.
$$
Thus the negative directions of $q_T$ correspond to the positive eigenvalues of $L$. If $Lf=\lambda f$, writing $f=uv$ and applying the maximum principle at a positive maximum of $v$ gives $\lambda\leq1$, with equality only when $f$ is a multiple of $u$. On the other hand, an elementary two-dimensional determinant inequality gives
$$
\|Lf\|^2\geq\langle f,Lf\rangle.
$$
For an eigenfunction this becomes $\lambda^2\geq\lambda$, so every positive eigenvalue satisfies $\lambda\geq1$. It follows that $1$ is the only positive eigenvalue of $L$, and that it is simple. Therefore $q_T$ has index one. The main theorem follows.

A different connection between minimal surface theory and convex geometry appears in the work of Matsumoto \cite{Mat25}.

\subsection*{AI disclosure:} These results were obtained by combining ideas of the authors with a publicly available LLM. This resulted in a proof of the main result which was then significantly modified by the authors (again using LLM assistance) into the present form. This paper does not contain AI written text. 

\subsection*{Acknowledgements}

O.C. was partially supported by a Terman Fellowship and an NSF grant (DMS-2304432).

\section{Preliminaries}

Throughout, $\BB^3\subset\R^3$ is the closed unit ball and $S^2=\partial\BB^3$. Write $X:\Sigma\to\BB^3$ for the inclusion. Properly embedded means that $X^{-1}(S^2)=\partial\Sigma$. The free-boundary condition says that $\Sigma$ meets $S^2$ orthogonally, and hence
$$
\eta=X\qquad\text{on }\partial\Sigma.
$$
Minimality is equivalent to $\Delta X=0$. Thus
\begin{equation}\label{eq:fbms}
\Delta X=0\quad\text{in }\Sigma,\qquad \partial_\eta X=X\quad\text{on }\partial\Sigma.
\end{equation}

The Steklov eigenvalues form a discrete sequence $0=\sigma_0<\sigma_1\leq\sigma_2\leq\cdots$. The form $Q$ introduced above is continuous on $H^1(\Sigma)$ by the trace theorem. For a symmetric form $B$, let $\operatorname{ind}(B)$ be the largest dimension of a subspace on which $B$ is negative definite. Since
$$
Q(1)=-|\partial\Sigma|<0,
$$
and $Q(1,f)=0$ whenever $\int_{\partial\Sigma}f\,ds=0$, the bound $\operatorname{ind}(Q)\leq1$ implies that $Q(f)\geq0$ on the boundary-mean-zero subspace. The Rayleigh characterization then gives $\sigma_1(\Sigma)\geq1$.

Writing $X=(X_1,X_2,X_3)$, equation \eqref{eq:fbms} shows that every nonzero coordinate function is a Steklov eigenfunction with eigenvalue $1$. Moreover,
$$
\int_{\partial\Sigma}X_i\,ds=\int_{\partial\Sigma}\partial_\eta X_i\,ds=\int_\Sigma\Delta X_i\,dA=0.
$$
Consequently,
\begin{equation}\label{eq:upper}
\sigma_1(\Sigma)\leq1.
\end{equation}
The only geometric input needed in the proof is the following. Its interior statement is due to Kusner--McGrath \cite[Corollary~4.2(i)]{KusnerMcGrath2024}; we include the short argument extending injectivity to the boundary in Appendix~A. (For $\Sigma$ satisfying $\sigma_1(\Sigma) =1$ this result was previously proven by Fraser--Schoen \cite[Proposition 8.1]{FraserSchoen2016}). 

\begin{proposition}\label{prop:radial}
If $\Sigma$ is an embedded genus zero free boundary minimal surface other than the equatorial disk, then $0\notin\Sigma$ and
$$
n=\frac{X}{|X|}:\Sigma\longrightarrow S^2
$$
is injective on $\Sigma$ and a local diffeomorphism on $\Sigma^\circ$.
\end{proposition}

If $\partial\Sigma$ has one component, then $\Sigma$ is a disk and hence, by Nitsche's theorem \cite{Nitsche1985}, an equatorial disk. This case is immediate, since its Steklov spectrum begins $0,1,1,2,2,\ldots$. We may therefore assume that $\partial\Sigma$ has at least two components. 

Henceforth $\Sigma$ is not an equatorial disk, and we write
$$
r=|X|,\qquad n=\frac{X}{|X|}.
$$

\section{Proof of the main theorem}

By the discussion above, it remains to prove
$$
\operatorname{ind}(Q)\leq 1.
$$

\begin{lemma}\label{lem:radial}
For every $v,w\in C^\infty(\Sigma)$,
\begin{equation}\label{eq:radial}
Q(rv,rw)=\int_\Sigma r^2\left(\langle\nabla v,\nabla w\rangle-|dn|^2vw\right)dA.
\end{equation}
\end{lemma}

\begin{proof}
Using $X=rn$, $|n|=1$, $\Delta X=0$, and $\partial_\eta X=X$, we obtain
$$
0=\langle\Delta(rn),n\rangle=\Delta r-r|dn|^2
$$
and
$$
 r=1,\qquad \partial_\eta r=\frac{\langle X,\partial_\eta X\rangle}{r}=1\quad\text{on }\partial\Sigma.
 $$
Consequently,
$$
\begin{aligned}
Q(rv,rw)
&=\int_\Sigma\left(r^2\langle\nabla v,\nabla w\rangle+\langle\nabla r,\nabla(rvw)\rangle\right)dA-\int_{\partial\Sigma}r^2vw\,ds\\
&=\int_\Sigma\left(r^2\langle\nabla v,\nabla w\rangle-rvw\,\Delta r\right)dA+\int_{\partial\Sigma}(r\partial_\eta r-r^2)vw\,ds\\
&=\int_\Sigma r^2\left(\langle\nabla v,\nabla w\rangle-|dn|^2vw\right)dA.
\end{aligned}
$$
This completes the proof.
\end{proof}
We now rewrite the right-hand side as a quadratic form on $S^2$.

Let $g$ and $d\omega$ denote the round metric and area measure on $S^2$, and set $\Omega=n(\Sigma^\circ)$. By Proposition~\ref{prop:radial}, the map $n:\Sigma^\circ\to\Omega$ is a diffeomorphism. For $y=n(p)$, let $J(p)>0$ be the Jacobian of $n$ and let $e_1,e_2$ be an orthonormal basis of $T_p\Sigma$. Define
\begin{equation}\label{eq:T}
T_y=\frac{r(p)^2}{J(p)}\sum_{\alpha=1}^2dn_p(e_\alpha)\otimes dn_p(e_\alpha),
\end{equation}
and extend $T$ by zero outside $\Omega$.

\begin{proposition}\label{prop:transfer}
The tensor $T$ is well defined, symmetric, nonnegative, and belongs to $L^1(S^2)$. Moreover, it is positive definite on $\Omega$. Furthermore,
\begin{equation}\label{eq:transfer}
q_T(\phi,\psi):=\int_{S^2}\left(\langle T\nabla\phi,\nabla\psi\rangle-(\tr T)\phi\psi\right)d\omega=Q\bigl(r(\phi\circ n),r(\psi\circ n)\bigr)
\end{equation}
for every $\phi,\psi\in C^\infty(S^2)$.
\end{proposition}

\begin{proof}
The sum in \eqref{eq:T} is $dn_p\,dn_p^*$, so $T$ is independent of the choice of orthonormal basis, symmetric, and nonnegative. Since $dn_p$ is an isomorphism, $T$ is positive definite on $\Omega$. For $y=n(p)$ the chain rule gives
\begin{equation*}
J(p)\langle T_y\nabla\phi,\nabla\psi\rangle
=r(p)^2\sum_{\alpha=1}^2e_\alpha(\phi\circ n)e_\alpha(\psi\circ n)
=r(p)^2\langle\nabla(\phi\circ n),\nabla(\psi\circ n)\rangle,
\end{equation*}
\begin{equation}\label{eq:traceT}
J(p)\tr T_y
=r(p)^2\sum_{\alpha=1}^2|dn_p(e_\alpha)|^2
=r(p)^2|dn_p|^2.
\end{equation}

Changing variables by $n$ therefore gives
$$
\begin{aligned}
q_T(\phi,\psi)
&=\int_\Sigma r^2\left(\langle\nabla(\phi\circ n),\nabla(\psi\circ n)\rangle-|dn|^2(\phi\circ n)(\psi\circ n)\right)dA\\
&=Q\bigl(r(\phi\circ n),r(\psi\circ n)\bigr),
\end{aligned}
$$
where the last equality follows from Lemma~\ref{lem:radial}. Finally, \eqref{eq:traceT} gives
$$
\int_{S^2}\tr T\,d\omega=\int_\Sigma r^2|dn|^2\,dA<\infty,
$$
and hence $T\in L^1(S^2)$.
\end{proof}

We shall use the following density statement, whose proof is given in Appendix~\ref{appendix}.

\begin{lemma}\label{lem:density}
The functions $r(\phi\circ n)$, with $\phi\in C^\infty(S^2)$, are dense in $H^1(\Sigma)$.
\end{lemma}

In the sequel we regard $q_T$ as a form on $C^\infty(S^2)$ so $\operatorname{ind}(q_T)$ is the index as a form on this vector space. (Note that since $T$ is a priori only in $L^1$ we may not be able to evaluate $q_T$ on arbitrary elements in $H^1(S^2)$. We avoid this issue in several places by smoothing arguments.)

\begin{lemma}\label{lem:index-transfer} $\operatorname{ind}(Q)\leq\operatorname{ind}(q_T).$
\end{lemma}

\begin{proof}
Let $E\subset H^1(\Sigma)$ be an $m$-dimensional subspace on which $Q$ is negative definite, and choose a basis $f_1,\ldots,f_m$ of $E$. By Lemma~\ref{lem:density}, there are functions $\phi_{i,k}\in C^\infty(S^2)$ such that
$$
v_{i,k}=r(\phi_{i,k}\circ n)\longrightarrow f_i\qquad\text{in }H^1(\Sigma).
$$
Since $Q$ is continuous, Proposition~\ref{prop:transfer} gives
$$
\bigl(q_T(\phi_{i,k},\phi_{j,k})\bigr)_{i,j=1}^m
=
\bigl(Q(v_{i,k},v_{j,k})\bigr)_{i,j=1}^m
\longrightarrow
\bigl(Q(f_i,f_j)\bigr)_{i,j=1}^m.
$$
The matrix on the right is negative definite, and hence so is the matrix on the left for all sufficiently large $k$. Thus $q_T$ has an $m$-dimensional negative subspace.
\end{proof}
We will use the notation
$$A:B:= \sum\limits_{i,j}A_{ij}B_{ij}.$$
and
$$(V\otimes W)_{ij}=V_iW_j.$$
We now have 
\begin{lemma}\label{lem:divfree} $\operatorname{div}T=0$ distributionally.
\end{lemma}

\begin{proof}
Let $x_1,x_2,x_3$ be the coordinate functions on $S^2$. Since $r(x_i\circ n)=X_i$, Proposition~\ref{prop:transfer} and \eqref{eq:fbms} give
\[
q_T(x_i,\phi)=Q\bigl(X_i,r(\phi\circ n)\bigr)=0
\]
for every $\phi\in C^\infty(S^2)$. Since $\nabla^2x_i=-x_i g$, the definition of $q_T$ and integration by parts yield
\[
0=q_T(x_i,\phi)=\int_{S^2}T:\nabla(\phi\nabla x_i)\,d\omega=-\langle\operatorname{div}T,\phi\nabla x_i\rangle.
\]
Thus, for every $V\in C^\infty(TS^2)$, the identity $\sum_{i=1}^3\nabla x_i\otimes\nabla x_i=g$ gives
\[
\langle\operatorname{div}T,V\rangle=\sum_{i=1}^3\bigl\langle\operatorname{div}T,\langle V,\nabla x_i\rangle\nabla x_i\bigr\rangle=0.
\]
Hence $\operatorname{div}T=0$ distributionally.
\end{proof}

We have therefore associated to $\Sigma$ a nonzero symmetric, nonnegative, divergence-free tensor $T\in L^1(S^2)$ such that
$$
\operatorname{ind}(Q)\leq\operatorname{ind}(q_T).
$$
It remains to prove that $q_T$ has index at most one.
\subsection{The spherical index bound} The goal of this section is to give a proof of the following result.
\begin{theorem}\label{thm:spherical} Let $T$ be the tensor defined in \eqref{eq:T}. Then $\operatorname{ind}(q_T)=1$. 
\end{theorem}

Using a smoothing argument (cf.\ \Cref{lem:tensor-approximation}), this follows from Minkowski's quadratic inequality in dimension three \cite{Minkowski} (cf.\ \Cref{rem:minkowski}). We give a proof below following Shenfeld--van Handel \cite[Section 3]{ShenfeldVanHandel2019} (see also \cite{CorderoErausquinKlartagMerigotSantambrogio2019} and \cite{ShenfeldVanHandel2022}). We note that the result when $T$ is smooth and positive-definite is also a direct consequence of Colesanti's result \cite[Theorem 4]{Colesanti2008}.

\begin{proposition}\label{pro:smooth-spherical} Let $T$ be a smooth positive definite symmetric tensor on $S^2$ satisfying $\operatorname{div}T=0$. Then $q_T$ has index one. 
\end{proposition}

Set
\[
Af=\nabla^2f+fg.
\]
Recall that a symmetric tensor $S$ is Codazzi if $\nabla_iS_{jk}=\nabla_jS_{ik}$. On $S^2$, $Af$ is Codazzi for every $f\in C^\infty(S^2)$, and in dimension two $S$ is Codazzi if and only if $\operatorname{cof}S$ is divergence-free, where 
$$
\operatorname{cof}S=(\tr S)g-S.
$$
We now have the following standard result (cf.\ \cite{Ferus,SimonOliker}):
\begin{lemma}\label{lem:potential} Let $S$ be a smooth positive definite symmetric Codazzi tensor on $S^2$. Then $S=Au$ for some $u\in C^\infty(S^2)$ with $u>0$. \end{lemma}

\begin{proof}
Regard $S$ as the $\R^3$-valued one-form $\alpha(Y)=S(Y)$. The Gauss formula, the Codazzi equation, and the symmetry of $S$ give
\[
d\alpha(Y,Z)=(\nabla_YS)(Z)-(\nabla_ZS)(Y)+\bigl(\langle S(Y),Z\rangle-\langle S(Z),Y\rangle\bigr)x=0.
\]
Since $S^2$ is simply connected, there is a smooth map $F:S^2\to\R^3$ such that $dF=S$. After translating $F$,
\[
\int_{S^2}F\,d\omega=0.
\]
Set $u(x)=F(x)\cdot x$. Since $dF(Y)$ is tangent, $du(Y)=F\cdot Y$, and hence
\[
F=\nabla u+ux.
\]
Differentiating in $\R^3$ gives
\[
dF(Y)=\nabla_Y\nabla u+uY,
\]
so $Au=S$.

It remains to prove that $u>0$. Fix $q\in S^2$. For $p\ne q$, let $\sigma:[0,\theta]\to S^2$ be a minimizing unit-speed geodesic from $p$ to $q$, where $\theta=d(p,q)$. The tangential projection of $q$ at $\sigma(t)$ is $\sin(\theta-t)\dot\sigma(t)$, so
\[
u(q)-F(p)\cdot q=(F(q)-F(p))\cdot q=\int_0^\theta\sin(\theta-t)S_{\sigma(t)}(\dot\sigma,\dot\sigma)\,dt>0.
\]
Integrating in $p$ and using the normalization of $F$ gives
\[
4\pi u(q)=\int_{S^2}\bigl(u(q)-F(p)\cdot q\bigr)\,d\omega(p)>0.
\]
This completes the proof. 
\end{proof}
We can now prove the main result:
\begin{proof}[Proof of \Cref{pro:smooth-spherical}]
Since $\operatorname{div}T=0$,
\[
\nabla_i(\operatorname{cof}T)_{jk}-\nabla_j(\operatorname{cof}T)_{ik}=g_{jk}(\operatorname{div}T)_i-g_{ik}(\operatorname{div}T)_j=0.
\]
Thus $\operatorname{cof}T$ is positive definite and Codazzi. By Lemma~\ref{lem:potential}, there is a smooth function $u>0$ such that $Au=\operatorname{cof}T$.

Set
\[
d\mu=\frac{2\det T}{u}\,d\omega,\qquad Lf=\frac{u}{2\det T}\,T:Af.
\]
Integration by parts gives
\[
q_T(f,\psi)=-\int_{S^2}f\,T:A\psi\,d\omega=-\langle f,L\psi\rangle_{L^2(\mu)}.
\]
Thus $L$ is elliptic and self-adjoint. Moreover, $T:\operatorname{cof}T=2\det T$ gives $Lu=u$, and the index of $q_T$ is therefore the total multiplicity of the positive eigenvalues of $L$.

Let $Lf=\lambda f$ with $\lambda>0$ and write $f=uv$. After replacing $f$ by $-f$, we may assume that $\max_{S^2}v>0$. Expanding $A(uv)$ and using $Au=\operatorname{cof}T$ gives
\[
u\,T:\nabla^2v+2T(\nabla u,\nabla v)=2(\lambda-1)(\det T)v.
\]
At a positive maximum of $v$, the left-hand side is nonpositive, and hence $\lambda\leq1$. If $\lambda=1$, the strong maximum principle gives that $v$ is constant. Thus every positive eigenvalue is at most $1$, and the $1$-eigenspace is spanned by $u$.

For every smooth $f$, the tensor $Af$ is Codazzi, so $\operatorname{cof}(Af)$ is divergence-free. Using $\operatorname{cof}B:\operatorname{cof}C=B:C$ and $\operatorname{cof}B:B=2\det B$, together with integration by parts, we obtain
\[
\begin{aligned}
\langle f,Lf\rangle_{L^2(\mu)}
&=\int_{S^2}f\,T:Af\,d\omega\\
&=\int_{S^2}f\,\operatorname{cof}(Af):Au\,d\omega\\
&=\int_{S^2}u\,\operatorname{cof}(Af):Af\,d\omega\\
&=2\int_{S^2}u\det(Af)\,d\omega.
\end{aligned}
\]
If $B$ is symmetric and $a,b$ are the eigenvalues of $T^{1/2}BT^{1/2}$, then
\begin{equation}\label{eq:matrix-ineq-TB}
    (T:B)^2-4\det T\det B=(a-b)^2\geq0.
\end{equation}
Indeed $T^{1/2}BT^{1/2}$ is symmetric so $a,b\in \RR$. Since trace is cyclic we find
\[
a+b=\tr (T^{1/2} B T^{1/2}) = \tr(TB )  = T:B. 
\]
Similarly using the multiplicative property of determinant we have
\[
ab = \det T \det B. 
\]

Applying \eqref{eq:matrix-ineq-TB} with $B=Af$ yields
\[
\begin{aligned}
\|Lf\|_{L^2(\mu)}^2-\langle f,Lf\rangle_{L^2(\mu)}
&=\int_{S^2}\frac{u}{2\det T}\left((T:Af)^2-4\det T\det(Af)\right)\,d\omega\\
&\geq0.
\end{aligned}
\]
For an eigenfunction with eigenvalue $\lambda>0$, this gives
\[
0\leq(\lambda^2-\lambda)\|f\|_{L^2(\mu)}^2,
\]
and hence $\lambda\geq1$. Therefore $1$ is the unique positive eigenvalue of $L$ and is simple, so $q_T$ has index one.
\end{proof}

\begin{remark}\label{rem:minkowski}
We recall here an alternative approach from \cite[Remark 3.2]{ShenfeldVanHandel2019}. Set $B=-q_T$, and let
$K$ have support function $u$ ($K$ may be obtained as the convex hull of $F(S^2)$ for $F$ defined as in \Cref{lem:potential}). For support functions $h,k$, the mixed-volume
formula gives
\[
B(h,k)=6V(K_h,K_k,K),
\]
where $K_h$ denotes the body with support function $h$. Thus
$B(u,u)=6\operatorname{Vol}(K)>0$. If $B(f,u)=0$, choose $c\gg1$ so that
$h=f+cu$ is a support function. Minkowski's quadratic inequality gives
\[
B(h,u)^2\geq B(h,h)B(u,u).
\]
Expanding gives $B(f,f)\leq0$. Hence $B$ has one positive direction, and $q_T$ has index one.
\end{remark}

\begin{lemma}\label{lem:tensor-approximation} Let $T\in L^1(S^2)$ be symmetric, nonnegative, and distributionally divergence-free. Then there are smooth positive definite symmetric divergence-free tensors $T_\varepsilon$ such that $T_\varepsilon\to T$ in $L^1(S^2)$. 
\end{lemma}

\begin{proof} 
Let $\eta_\varepsilon$ be a nonnegative smooth approximate identity on $\mathrm{SO}(3)$ with respect to Haar probability measure. For $R\in\mathrm{SO}(3)$, set $(R_*T)_x=R\,T_{R^{-1}x}R^{-1}$ and define
\[
T_\varepsilon=\int_{\mathrm{SO}(3)}R_*T\,\eta_\varepsilon(R)\,dR+\varepsilon g.
\]
The convolution is smooth, rotations preserve symmetry, nonnegativity, and divergence-freeness, and $\varepsilon g$ makes $T_\varepsilon$ positive definite. Moreover,
\[
\|T_\varepsilon-T\|_{L^1}\leq\int_{\mathrm{SO}(3)}\|R_*T-T\|_{L^1}\eta_\varepsilon(R)\,dR+\varepsilon\|g\|_{L^1}\to 0.
\]
This completes the proof. 
\end{proof}

\begin{proof}[Proof of \Cref{thm:spherical}] By Proposition~\ref{prop:transfer} and Lemma~\ref{lem:divfree}, the tensor $T\in L^1(S^2)$ is symmetric, nonnegative, nonzero, and distributionally divergence-free. Let $T_\varepsilon$ be given by Lemma~\ref{lem:tensor-approximation}. By Proposition~\ref{pro:smooth-spherical},
\[
\operatorname{ind}(q_{T_\varepsilon})=1.
\]

Suppose that $q_T$ is negative definite on a two-dimensional subspace $E\subset C^\infty(S^2)$, and let $f_1,f_2$ be a basis of $E$. Since $T_\varepsilon\to T$ in $L^1(S^2)$,
\[
\bigl(q_{T_\varepsilon}(f_i,f_j)\bigr)_{i,j=1}^2\longrightarrow\bigl(q_T(f_i,f_j)\bigr)_{i,j=1}^2.
\]
The matrix on the right is negative definite, so the matrix on the left is negative definite for sufficiently small $\varepsilon$, contradicting $\operatorname{ind}(q_{T_\varepsilon})=1$. Hence $\operatorname{ind}(q_T)\leq1$.

Finally, since $T$ is nonnegative and nonzero,
\[
q_T(1)=-\int_{S^2}\tr T\,d\omega<0.
\]
Thus $\operatorname{ind}(q_T)=1$. 
\end{proof}

\subsection{Conclusion}

\begin{proof}[Proof of Theorem~\ref{thm:main}] By Lemma~\ref{lem:index-transfer} and Theorem~\ref{thm:spherical},
\[
\operatorname{ind}(Q)\leq\operatorname{ind}(q_T)=1.
\]

If $\int_{\partial\Sigma}f\,ds=0$, then $Q(1,f)=0$. Since $Q(1)=-|\partial\Sigma|<0$, we have $Q(f)\geq0$, for otherwise $Q$ would be negative definite on $\operatorname{span}\{1,f\}$. Thus the Rayleigh characterization gives $\sigma_1(\Sigma)\geq1$. Together with \eqref{eq:upper}, this yields
\[
\sigma_1(\Sigma)=1.
\]
This completes the proof. 
\end{proof}
\appendix

\section{The radial projection}\label{appendix}

\begin{proof}[Proof of Proposition~\ref{prop:radial}]
By Nitsche's theorem \cite{Nitsche1985}, $\partial\Sigma$ has at least two components. It follows from \cite[Corollary~4.2(i)]{KusnerMcGrath2024} that $0\notin\Sigma$ and every ray from the origin meets $\Sigma^\circ$ transversely in at most one point. This property gives injectivity of $n$ on $\Sigma^\circ$, while transversality and $\ker dn_p=T_p\Sigma\cap\R X(p)$ show that $n$ is a local diffeomorphism there.

Since $n=X$ on $\partial\Sigma$, it is injective on $\partial\Sigma$. Suppose that $n(p)=n(q)$ for some $p\in\Sigma^\circ$ and $q\in\partial\Sigma$. Choose a neighborhood $U$ of $p$ with $\overline{U}\subset\Sigma^\circ$ such that $n|_U:U\to W$ is a diffeomorphism onto an open neighborhood $W$ of $n(p)$. Choose $q_j\in\Sigma^\circ\setminus\overline{U}$ with $q_j\to q$. Since $n(q_j)\to n(q)=n(p)$, we have $n(q_j)\in W$ for large $j$, and hence $p_j=(n|_U)^{-1}(n(q_j))\in U$ satisfies $n(p_j)=n(q_j)$. This contradicts injectivity on $\Sigma^\circ$, and therefore $n$ is injective on $\Sigma$.
\end{proof}

\begin{proof}[Proof of Lemma~\ref{lem:density}]
It is enough to approximate $f\in C^\infty(\Sigma)$. Let $\Gamma=n(\partial\Sigma)$. Since $n=X$ on $\partial\Sigma$, the restriction $n|_{\partial\Sigma}$ is a diffeomorphism onto the smooth embedded $1$-manifold $\Gamma$. Thus $f\circ(n|_{\partial\Sigma})^{-1}$ extends to some $\phi_0\in C^\infty(S^2)$. Since $r=1$ on $\partial\Sigma$, the function $h=f-r(\phi_0\circ n)$ vanishes on $\partial\Sigma$.

Let $d$ be the distance to $\partial\Sigma$ in a boundary collar, and choose $\chi_\varepsilon\in C^\infty(\Sigma)$ equal to $0$ where $d\leq\varepsilon$ and to $1$ outside $\{d<2\varepsilon\}$. Since $h=O(d)$ and $|\nabla\chi_\varepsilon|=O(\varepsilon^{-1})$, the functions $h_\varepsilon=\chi_\varepsilon h$ belong to $C_c^\infty(\Sigma^\circ)$ and satisfy $h_\varepsilon\to h$ in $H^1(\Sigma)$.

Set $\Omega=n(\Sigma^\circ)$. By Proposition~\ref{prop:radial}, $r>0$ on $\Sigma$ and $n:\Sigma^\circ\to\Omega$ is a diffeomorphism. Hence $(h_\varepsilon/r)\circ n^{-1}\in C_c^\infty(\Omega)$, and its extension by zero defines a function $\phi_\varepsilon\in C^\infty(S^2)$. Since $n$ is injective on $\Sigma$, we have $\Gamma\cap\Omega=\varnothing$, and therefore $r(\phi_\varepsilon\circ n)=h_\varepsilon$ on $\Sigma$. Consequently,
\[
r\bigl((\phi_0+\phi_\varepsilon)\circ n\bigr)=r(\phi_0\circ n)+h_\varepsilon\longrightarrow f
\]
in $H^1(\Sigma)$.
\end{proof}

\printbibliography

@article{LimaMenezes2021,
  author  = {Lima, Vanderson and Menezes, Ana},
  title   = {A two-piece property for free boundary minimal surfaces in the ball},
  journal = {Transactions of the American Mathematical Society},
  volume  = {374},
  number  = {3},
  pages   = {1661--1686},
  year    = {2021},
  doi     = {10.1090/tran/8223}
}

@article{McGrath2018,
  author  = {McGrath, Peter},
  title   = {A characterization of the critical catenoid},
  journal = {Indiana University Mathematics Journal},
  volume  = {67},
  number  = {2},
  pages   = {889--897},
  year    = {2018},
  doi     = {10.1512/iumj.2018.67.7251}
}

@article{McGrathZou2024,
  author  = {McGrath, Peter and Zou, Jiahua},
  title   = {On the areas of genus zero free boundary minimal surfaces embedded in the unit $3$-ball},
  journal = {The Journal of Geometric Analysis},
  volume  = {34},
  number  = {9},
  eid     = {274},
  year    = {2024},
  doi     = {10.1007/s12220-024-01726-2}
}

@misc{PZ26,
  author = {Davide Parise and Jonathan J. Zhu},
  title = {Uniqueness of free boundary minimal annuli},
  year = {2026},
  eprint = {2609.17512},
  archivePrefix = {arXiv},
  primaryClass = {math.DG}
}

@misc{Mat25,
  author = {Jun Matsumoto},
  title = {A class of affine maximal surfaces with
           singularities and its relationship
           with minimal surface theory},
  year = {2025},
  eprint = {2507.10035},
  archivePrefix = {arXiv},
  primaryClass = {math.DG}
}

@article{Colesanti2008,
  author  = {Colesanti, Andrea},
  title   = {From the {Brunn--Minkowski} inequality to a class of {Poincar\'e}-type inequalities},
  journal = {Communications in Contemporary Mathematics},
  volume  = {10},
  number  = {5},
  pages   = {765--772},
  year    = {2008},
  doi     = {10.1142/S0219199708002971}
}

@misc{Seo2024,
  author        = {Seo, Dong-Hwi},
  title         = {A symmetry condition for genus zero free boundary minimal surfaces attaining the first eigenvalue of one},
  year          = {2024},
  eprint        = {2409.14718},
  archivePrefix = {arXiv},
  primaryClass  = {math.DG}
}

@article{Seo2026,
  author  = {Seo, Dong-Hwi},
  title   = {Sufficient symmetry conditions for free boundary minimal annuli to be the critical catenoid},
  journal = {Selecta Mathematica. New Series},
  volume  = {32},
  number  = {3},
  eid     = {42},
  year    = {2026},
  doi     = {10.1007/s00029-026-01150-0}
}

@article{CarlottoFranzSchulz2022,
  author  = {Carlotto, Alessandro and Franz, Giada and Schulz, Mario B.},
  title   = {Free boundary minimal surfaces with connected boundary and arbitrary genus},
  journal = {Cambridge Journal of Mathematics},
  volume  = {10},
  number  = {4},
  pages   = {835--857},
  year    = {2022},
  doi     = {10.4310/CJM.2022.v10.n4.a3}
}

@article{CarlottoSchulzWiygul2025a,
  author  = {Carlotto, Alessandro and Schulz, Mario B. and Wiygul, David},
  title   = {Infinitely many pairs of free boundary minimal surfaces with the same topology and symmetry group},
  journal = {Memoirs of the American Mathematical Society},
  volume  = {313},
  number  = {1591},
  year    = {2025},
  doi     = {10.1090/memo/1591}
}

@article{CarlottoSchulzWiygul2025b,
  author  = {Carlotto, Alessandro and Schulz, Mario B. and Wiygul, David},
  title   = {Disc stackings and their {Morse} index},
  journal = {Advanced Nonlinear Studies},
  volume  = {25},
  number  = {3},
  year    = {2025},
  doi     = {10.1515/ans-2023-0177}
}

@article{Chu2023,
  author  = {Chu, Adrian Chun-Pong},
  title   = {A free boundary minimal surface via a $6$-sweepout},
  journal = {The Journal of Geometric Analysis},
  volume  = {33},
  number  = {7},
  eid     = {230},
  year    = {2023},
  doi     = {10.1007/s12220-023-01285-y}
}

@article{FernandezHauswirthMira2023,
  author  = {Fern{\'a}ndez, Isabel and Hauswirth, Laurent and Mira, Pablo},
  title   = {Free Boundary Minimal Annuli Immersed in the Unit Ball},
  journal = {Archive for Rational Mechanics and Analysis},
  volume  = {247},
  eid     = {108},
  year    = {2023},
  doi     = {10.1007/s00205-023-01943-z}
}

@article{FolhaPacardZolotareva2017,
  author  = {Folha, Abigail and Pacard, Frank and Zolotareva, Tatiana},
  title   = {Free boundary minimal surfaces in the unit $3$-ball},
  journal = {Manuscripta Mathematica},
  volume  = {154},
  pages   = {359--409},
  year    = {2017},
  doi     = {10.1007/s00229-017-0924-9}
}

@misc{FranzKetoverSchulz2024,
  author        = {Franz, Giada and Ketover, Daniel and Schulz, Mario B.},
  title         = {Genus one critical catenoid},
  year          = {2024},
  eprint        = {2409.12588},
  archivePrefix = {arXiv},
  primaryClass  = {math.DG}
}

@article{FranzSchulz2026,
  author  = {Franz, Giada and Schulz, Mario B.},
  title   = {Topological control for min-max free boundary minimal surfaces},
  journal = {Journal of the European Mathematical Society},
  year    = {2026},
  doi     = {10.4171/JEMS/1800},
  note    = {Published online first}
}

@article{FraserLi2014,
  author  = {Fraser, Ailana and Li, Martin Man-Chun},
  title   = {Compactness of the space of embedded minimal surfaces with free boundary in three-manifolds with nonnegative {Ricci} curvature and convex boundary},
  journal = {Journal of Differential Geometry},
  volume  = {96},
  number  = {2},
  pages   = {183--200},
  year    = {2014},
  doi     = {10.4310/jdg/1393424916}
}

@article{FraserSchoen2011,
  author  = {Fraser, Ailana and Schoen, Richard},
  title   = {The first {Steklov} eigenvalue, conformal geometry, and minimal surfaces},
  journal = {Advances in Mathematics},
  volume  = {226},
  number  = {5},
  pages   = {4011--4030},
  year    = {2011},
  doi     = {10.1016/j.aim.2010.11.007}
}

@article{FraserSchoen2016,
  author  = {Fraser, Ailana and Schoen, Richard},
  title   = {Sharp eigenvalue bounds and minimal surfaces in the ball},
  journal = {Inventiones Mathematicae},
  volume  = {203},
  number  = {3},
  pages   = {823--890},
  year    = {2016},
  doi     = {10.1007/s00222-015-0604-x}
}

@article{KapouleasLi2021,
  author  = {Kapouleas, Nikolaos and Li, Martin Man-chun},
  title   = {Free boundary minimal surfaces in the unit three-ball via desingularization of the critical catenoid and the equatorial disk},
  journal = {Journal f{\"u}r die reine und angewandte Mathematik},
  volume  = {776},
  pages   = {201--254},
  year    = {2021},
  doi     = {10.1515/crelle-2020-0050}
}

@article{KapouleasMcGrath2023,
  author  = {Kapouleas, Nikolaos and McGrath, Peter},
  title   = {Generalizing the Linearized Doubling approach, {I}: General theory and new minimal surfaces and self-shrinkers},
  journal = {Cambridge Journal of Mathematics},
  volume  = {11},
  number  = {2},
  pages   = {299--439},
  year    = {2023},
  doi     = {10.4310/CJM.2023.v11.n2.a1}
}

@misc{KapouleasMcGrath2022,
  author        = {Kapouleas, Nikolaos and McGrath, Peter},
  title         = {Free Boundary Minimal Annuli Immersed in the Unit $3$-Ball},
  year          = {2022},
  eprint        = {2212.09680},
  archivePrefix = {arXiv},
  primaryClass  = {math.DG}
}

@article{KapouleasWiygul2023,
  author  = {Kapouleas, Nicolaos and Wiygul, David},
  title   = {Free-boundary minimal surfaces with connected boundary in the $3$-ball by tripling the equatorial disc},
  journal = {Journal of Differential Geometry},
  volume  = {123},
  number  = {2},
  pages   = {311--362},
  year    = {2023}
}

@article{KapouleasZou2026,
  author  = {Kapouleas, Nikolaos and Zou, Jiahua},
  title   = {Free boundary minimal surfaces in the Euclidean three-ball close to the boundary},
  journal = {American Journal of Mathematics},
  volume  = {148},
  number  = {3},
  year    = {2026},
  eprint  = {2111.11308},
  archivePrefix = {arXiv},
  primaryClass  = {math.DG}
}

@misc{KarpukhinKusnerMcGrathStern2024,
  author        = {Karpukhin, Mikhail and Kusner, Robert and McGrath, Peter and Stern, Daniel},
  title         = {Embedded minimal surfaces in $\mathbb S^3$ and $\mathbb B^3$ via equivariant eigenvalue optimization},
  year          = {2024},
  eprint        = {2402.13121},
  archivePrefix = {arXiv},
  primaryClass  = {math.DG}
}

@article{KarpukhinStern2024,
  author  = {Karpukhin, Mikhail and Stern, Daniel L.},
  title   = {From {Steklov} to {Laplace}: free boundary minimal surfaces with many boundary components},
  journal = {Duke Mathematical Journal},
  volume  = {173},
  number  = {8},
  pages   = {1557--1629},
  year    = {2024},
  doi     = {10.1215/00127094-2023-0041}
}

@misc{Ketover2016,
  author        = {Ketover, Daniel},
  title         = {Free boundary minimal surfaces of unbounded genus},
  year          = {2016},
  eprint        = {1612.08691},
  archivePrefix = {arXiv},
  primaryClass  = {math.DG}
}

@article{KusnerMcGrath2024,
  author  = {Kusner, Robert and McGrath, Peter},
  title   = {On {Steklov} eigenspaces for free boundary minimal surfaces in the unit ball},
  journal = {American Journal of Mathematics},
  volume  = {146},
  number  = {5},
  pages   = {1275--1293},
  year    = {2024},
  doi     = {10.1353/ajm.2024.a937942}
}

@article{CorderoErausquinKlartagMerigotSantambrogio2019,
  author  = {Cordero-Erausquin, Dario and Klartag, Bo'az and M{\'e}rigot, Quentin and Santambrogio, Filippo},
  title   = {One more proof of the {Alexandrov--Fenchel} inequality},
  journal = {Comptes Rendus Math{\'e}matique. Acad{\'e}mie des Sciences. Paris},
  volume  = {357},
  number  = {8},
  pages   = {676--680},
  year    = {2019},
  doi     = {10.1016/j.crma.2019.07.004}
}

@article{Nitsche1985,
  author  = {Nitsche, Johannes C. C.},
  title   = {Stationary partitioning of convex bodies},
  journal = {Archive for Rational Mechanics and Analysis},
  volume  = {89},
  number  = {1},
  pages   = {1--19},
  year    = {1985},
  doi     = {10.1007/BF00281743}
}

@article{ShenfeldVanHandel2019,
  author  = {Shenfeld, Yair and van Handel, Ramon},
  title   = {Mixed volumes and the {Bochner} method},
  journal = {Proceedings of the American Mathematical Society},
  volume  = {147},
  number  = {12},
  pages   = {5385--5402},
  year    = {2019},
  doi     = {10.1090/proc/14651}
}

@article{ShenfeldVanHandel2022,
  author  = {Shenfeld, Yair and van Handel, Ramon},
  title   = {The extremals of {Minkowski}'s quadratic inequality},
  journal = {Duke Mathematical Journal},
  volume  = {171},
  number  = {4},
  pages   = {957--1027},
  year    = {2022},
  doi     = {10.1215/00127094-2021-0033}
}

@article {Brendle:lawson,
    AUTHOR = {Brendle, Simon},
     TITLE = {Embedded minimal tori in {$S^3$} and the {L}awson conjecture},
   JOURNAL = {Acta Math.},
  FJOURNAL = {Acta Mathematica},
    VOLUME = {211},
      YEAR = {2013},
    NUMBER = {2},
     PAGES = {177--190},
      ISSN = {0001-5962,1871-2509},
   MRCLASS = {53A10 (49Q05 53C42)},
  MRNUMBER = {3143888},
MRREVIEWER = {Jo\~ao\ Lucas Marques Barbosa},
       DOI = {10.1007/s11511-013-0101-2},
       URL = {https://doi.org/10.1007/s11511-013-0101-2},
}

@article {Minkowski,
    AUTHOR = {Minkowski, Hermann},
     TITLE = {Volumen und {O}berfl\"ache},
   JOURNAL = {Math. Ann.},
  FJOURNAL = {Mathematische Annalen},
    VOLUME = {57},
      YEAR = {1903},
    NUMBER = {4},
     PAGES = {447--495},
      ISSN = {0025-5831,1432-1807},
   MRCLASS = {99-04},
  MRNUMBER = {1511220},
       DOI = {10.1007/BF01445180},
       URL = {https://doi.org/10.1007/BF01445180},
}

@InProceedings{Ferus,
author="Ferus, Dirk",
editor="Ferus, Dirk
and K{\"u}hnel, Wolfgang
and Simon, Udo
and Wegner, Bernd",
title="A remark on codazzi tensors in constant curvature spaces",
booktitle="Global Differential Geometry and Global Analysis",
year="1981",
publisher="Springer Berlin Heidelberg",
address="Berlin, Heidelberg",
pages="257--257",
isbn="978-3-540-38419-9"
}

@article{SimonOliker,
url = {https://doi.org/10.1515/crll.1983.342.35},
title = {Codazzi tensors and equations of Monge-Ampère type on compact manifolds of constant sectional curvature.},
title = {},
author = {U. SIMON and V.I. Oliker},
pages = {35--65},
volume = {1983},
number = {342},
journal = {Journal für die reine und angewandte Mathematik},
doi = {doi:10.1515/crll.1983.342.35},
year = {1983},
lastchecked = {2026-09-11}
}

\end{document}